\documentclass[11pt, a4paper]{amsart}

\usepackage[hypertexnames=false,breaklinks]{hyperref}
\usepackage{graphicx}
\usepackage{slashbox}
\usepackage{color}
\usepackage{nicefrac}
\usepackage{mathtools}
\usepackage{float} 

\renewcommand{\O}{\Omega}

\def\R{{\mathbb {R}}}

\def\lp{L^{p(\cdot)}(\Omega)}
\def\lq{L^{q(\cdot)}(\Omega)}

\def\lpe{L^{p^*(\cdot)}(\Omega)}
\def\wp{W^{1,p(\cdot)}(\Omega)}

\def\wpc{W_0^{1,p(\cdot)}(\Omega)}
\def\wpg{W_g^{1,p(\cdot)}(\Omega)}

\def\Lipc{\mathop{\mbox{\normalfont Lip}}\nolimits(\overline{\O})}

\DeclareMathOperator*{\esssup}{ess\,sup}
\DeclareMathOperator*{\essinf}{ess\,inf}

\newcommand{\de}{\coloneqq}

\def\T{\mathcal{T}}

\newtheorem{teo}{Theorem}[section]
\newtheorem{lema}[teo]{Lemma}
\newtheorem{prop}[teo]{Proposition}

\theoremstyle{definition}

\newtheorem{example}{Example}

\theoremstyle{remark}
\newtheorem{remark}[teo]{Remark}

\def\div{\mathop{\mbox{\normalfont div}}\nolimits}

\begin{document}

\title[Order of convergence of the fem for the $p(x)-$Laplacian]
{Order of convergence of the finite element method for the $p(x)-$Laplacian}

\author[L. M. Del Pezzo]{Leandro M. Del Pezzo}
\address{Leandro M. Del Pezzo \hfill\break\indent 
CONICET and Departamento  de 
Matem\'atica, FCEyN, UBA,\hfill\break\indent 
Pabell\'on I, Ciudad Universitaria (1428)\\
Buenos Aires, Argentina.}
\email{{\tt ldpezzo@dm.uba.ar}\hfill\break\indent {\it Web page:} 
\url{http://cms.dm.uba.ar/Members/ldpezzo}}
\thanks{L. M. Del Pezzo was partially supported by 
ANPCyT PICT-2012-0153, UBACyT 20020110300067 and 
CONICET PIP 11220090100643.}

\author[S. Mart\'{i}nez]{Sandra Mart\'{\i}nez}
\address{Sandra Mart\'{\i}nez \hfill\break\indent
IMAS-CONICET and Departamento  de Matem\'atica, FCEyN, UBA,
\hfill\break\indent Pabell\'on I, Ciudad Universitaria (1428),Buenos Aires, Argentina.}
\email{{\tt smartin@dm.uba.ar}}
\thanks{S. Mart\'inez was partially supported by 
	ANPCyT PICT-2012-0153, UBACyT 20020100100496 
	and CONICET PIP 11220090100625.}


\begin{abstract}
    In this work, we study the rate of convergence of the finite
    element method for the $p(x)-$Laplacian 
    ($1\leq p_1 \leq p(x)\leq p_2\leq 2$) in
    two dimensional convex domains. 
\end{abstract}

\maketitle

\section{Introduction}
Let $\Omega$ be a bounded convex domain in $\R^2$ with Lipschitz
boundary and $p:\Omega\to (1,+\infty)$ be a measurable function. 
In this work, we first consider the Dirichlet problem for the 
$p(x)-$Lapalacian
\begin{equation}
    \begin{cases}
       - \Delta_{p(x)} u = f & \mbox{in } \O,\\
         u = g & \mbox{on } \partial \Omega,
    \end{cases}
    \label{problema}
\end{equation}
where $\Delta_{p(x)}u=\div(|\nabla u|^{p(x)-2}\nabla u)$ is the 
$p(x)-$Laplacian and 
$|\cdot|^2=\langle\cdot,\cdot\rangle_{\R^2}$. 
The assumptions over $p$, $f$ and $g$  will be specified later.

\medskip

Note that, the $p(x)-$Laplacian extends the classical Laplacian 
($p(x)\equiv2$) and the $p-$Laplacian 
($p(x)\equiv p$ with $1<p<+\infty$). This operator has been
recently used in image processing and in the modeling of 
electrorheological fluids, see \cite{BCE,CLR,R}.

\medskip

A function $u\in \wpg\de\{v\in \wp\colon v=g \mbox{ on } 
\partial\O\}$ is a weak solution of \eqref{problema} 
if
\begin{equation}\label{weaksol}
	\int_{\O}|\nabla u|^{p(x)-2} 
	\nabla u \nabla v\, dx=\int_{\O} f v\,dx
\end{equation}
for all $ v\in\wpc.$

\medskip

Motivated by the applications to image processing problem, in
\cite{DLM}, the authors study the convergence of the discontinuous 
Galerking finite element method and the continuous Galerking 
finite element method (FEM) to approximate  weak
solutions of the equations of the type \eqref{problema}.
On the other hand, motivated by the application to electrorheological fluids, in \cite{CHP, PW} the authors prove weak 
convergence of an implicit finite element discretization for  a parabolic equation involving the $p(x)-$Laplacian. 

In \cite{DM}, we prove the $H^2$ regularity of the solution of
\eqref{weaksol} when $\Omega$ is a bounded domain with convex 
boundary and under certain assumptions for $p,f$ and $g$ (see
Section \ref{A} for details).

In the present work, we study the rate of convergence of the 
continuous Galerking FEM in the case where $p\colon\O\to[p_1,p_2]$ with $1<p_1\le p_2\le2$. To this end, we will follow the ideas of 
\cite{LB2,LB3,LB}, where the authors study the case $p(x)\equiv p$ 
($1<p<+\infty).$

\medskip

More precisely, let $h>0,$ $\O^h$ be a polygonal subset of $\O$ and 
$\T^h$ be a regular triangulation of $\O^h,$ where each triangle 
$\kappa\in\T^h$ has maximum diameter bounded by $h.$ Let 
$S^h$ denote the space of $C^0$ piecewise linear with respect to $\T^h.$
Our finite element approximation of \eqref{problema} is:  

\medskip

Find $u^h\in S^h_g$ such that
\begin{equation}\label{fem}
	\int_{\O^h}|\nabla u^h|^{p-2} \nabla u^h \nabla v dx 
	=\int_{\O^h} f v\,dx\quad \forall v\in S_0^h
\end{equation}
where
\[
	S_0^h\de\{v\in S^h\colon v=0\mbox{ on } \partial \O^h\},
	\quad
	S^h_g\de\{v\in S^h\colon 
	v=g^h \mbox{ on } \partial \O^h\},
\]
 and $g^h\in S^h$ is chosen to approximate the Dirichlet boundary data.

\medskip

In Theorem 7.2 in \cite{DLM}, the authors prove that if $p(x)$ is a 
$\log$-H\"older continous function (see Section  \ref{A} for the 
definition) the sequence of solutions of \eqref{fem} converge to the 
solution of \eqref{weaksol}.
In the present work, we  study the rate of convergence of this method. 
In general, all the error bounds depend on the global regularity 
of the second derivatives of the
solution. For example, in  the case $p(x)\equiv p,$
if $1<p\leq 2$ there exists a constant 
$C=C(\|u\|_{W^{2,p}(\O)})$ such that
\[
	\|u-u^h\|_{W^{1,p}(\O_h)}\leq C h^{\nicefrac{p}2},
\]
where $u\in W^{2,p}(\O)$ is the weak solution of \eqref{problema}
and $u^h$ is the solution of \eqref{fem}, 
see \cite{LB2}. Under more regularity assumptions over the function $u$ it was proved, in different works, 
optimal order of convergence (see for example \cite{LB2, EL, LB3}).

\bigskip

The main results of the present paper are the following theorems.

\begin{teo}\label{orden1} Let $p\colon\O\to[p_1,p_2]$ 
be a $\log$-H\"older continuous function with $1<p_1\le p_2\le2,$  
$f\in\lq$ with $q(x)\ge q_1>2,$
$g\in H^2(\O),$ $u$ and $u^h$ be the unique solutions of 
\eqref{weaksol} and 
\eqref{fem} respectively. Then
$$\|u-u^h\|_{W^{1,p(\cdot)}(\O_h)}\leq C h^{\nicefrac{p_1}2},$$
where $C$ is a constant that depends on $p(x)$, $\|f\|_{\lq}$ and  
$\|g\|_{H^2(\O)}$. 
\end{teo}

For sufficiently regular solutions, 
we obtain optimal order of convergence.

\begin{teo}\label{c2alphaloc}
Let $p\colon\O\to[p_1,p_2]$ be a $\log$-H\"older continuous function with $1<p_1\le p_2\le2,$ 
$u$ and $u^h$ be the unique solutions of \eqref{weaksol} and 
\eqref{fem} respectively. If
\begin{equation}\label{int11}
	\int_\O |\nabla u|^{p(x)-2}{H[u]^2}\,dx<+\infty
\end{equation}
where $H[u]=|u_{x_1 x_1}|+|u_{x_1 x_2}|+|u_{x_2 x_2}|$ and 
\begin{equation}\label{reghold}
	u\in C^{2,\alpha^+}(\tau) \mbox{ for each }  \tau\in \T^h 
\end{equation}
with $\alpha^+=\nicefrac{(2-p^+)}{p^+}$ and 
$p^+=\displaystyle\max_{x\in \tau} p(x),$
then
\[
	\|u-u^h\|_{1,p,\O^h}\le Ch.		
\]
\end{teo}

Finally, we show that if $\Omega$ is a ball, 
$p$ and $f$ are radially symmetric 
functions, $g$ is constant and  
\begin{equation}\label{conejemploint}
 	p \in C^{1,\beta}(\tau), f \in C^{\beta}(\tau) 
 	\mbox{ with }\beta\geq \alpha^+ \quad 
 	\forall \tau \in \T^h 
\end{equation}
then the assumptions of Theorem \ref{c2alphaloc} are satisfied. 
So in this case we have optimal order of convergence.
Observe that these regularity assumptions on the data are local, 
and depend only on $p^+$.

Note that, in order to have optimal order,
by \eqref{conejemploint},  we need
$p,f\in C^2$ in regions  where the maximum of 
$p$ is $2$, and  we also need, for example, $p,f\in C^{2,1}$ 
only in regions  where the function $p(x)$ is near $1$.

\bigskip

\noindent{\bf Organization of the paper}.  
In Section \ref{A} we collect some preliminary facts concerning
variable Sobolev spaces, the weak solution of \eqref{problema}, 
finite element spaces and Decomposition--Coordination method; in Section \ref{orden p1/2} we prove Theorem 
\ref{orden1},  Theorem \ref{c2alphaloc} and study the radially symmetric case, and finally in 
Section \ref{ejemplos} we show a family of numerical examples where 
we study the behaviour of the error when we use the 
Decomposition--Coordination method to approximate the solution
\eqref{fem}.

\section{Preliminaries}\label{A}

We begin with a review of the basic results that 
will be needed in subsequent sections.
The known results are generally stated without proofs, 
but we provide references where
the proofs can be found. Also, 
we introduce some of our notational conventions.

\subsection{General Properties of Variable Sobolev Spaces}

We first introduce the space  $L^{p(\cdot)}(\Omega)$ and 
$W^{1,p(\cdot)}(\Omega)$
and state some of their properties.

\medskip

Let $\O$ be a bounded open set of $\R^n$ and 
$p \colon\Omega \to  [1,+\infty]$ be a measurable bounded function,
called a variable exponent on $\Omega.$ Denote 
\[
	p_{1}\de \essinf_{x\in \O} p(x) \mbox{ and } 
	p_{2} \de \esssup_{x\in \O} \,p(x).
\]

We define the variable exponent Lebesgue space 
$\lp$ to consist of all measurable functions
$u \colon\Omega \to \R$ for which the modular
\[
	\varrho_{p(\cdot),\O}(u) \de 
	\int_{\Omega} \varphi(|u(x)|,p(x))\, dx
\]
is finite, where $\varphi\colon [0,+\infty)\times[1,+\infty]\to [0,+\infty]$
\[
	\varphi(t,p)=
		\begin{cases}
			t^{p} &\mbox{ if } p\neq \infty,\\
			\infty \chi_{(1,\infty)}(t) &\mbox{ if } p= \infty,
		\end{cases}
\]
with the notation $\infty\cdot 0=0.$

We define the Luxemburg norm on this space by
\[
	\|u\|_{p(\cdot),\Omega}\de 
	\inf\{k > 0\colon \varrho_{p(\cdot),\Omega}(u/
	k)\leq 1 \}.
\]
This norm makes $L^{p(\cdot)}(\Omega)$ a Banach space.

We will write it simply $\varrho_{p(\cdot)}(u)$ and
$\|u\|_{p(\cdot)}$  when no confusion can arise. 
\medskip

The following lemma can be found in \cite{LB}.

\begin{lema}\label{desigualdades} For any 
$p,\delta:\O\to \R_{\ge0}$ be measurable functions 
with $1<p_1\leq p(x)\leq p_2<+\infty,$  there exist
positive constants $C_1$ and $C_2$ (both depending on $p_1$ and $p_2$)  such that for all
$\xi,\eta\in\R^2,$ $\xi\neq\eta,$ $x\in\O$  we have
\begin{equation}\label{21a}
	||\xi|^{p(x)-2}\xi-|\eta|^{p(x)-2}\eta|
	\le C_1|\xi-\eta|^{1-\delta(x)}
	(|\xi|+|\eta|)^{p(x)-2+\delta(x)},	
\end{equation}
and
\begin{equation}
		\label{21b}
		(|\xi|^{p(x)-2}\xi-|\eta|^{p(x)-2}\eta)(\xi-\eta)
		\ge C_2 |\xi-\eta|^{2+\delta(x)}
		(|\xi|+|\eta|)^{p(x)-2-\delta(x)}.	
\end{equation}
\end{lema}

\medskip

For the proofs of the following theorems, we refer the reader to
\cite{DHHR}.

\begin{lema}\label{compinf}
    Let $p:\O\to[1,+\infty]$ be a measurable function  with $p_1<\infty$. If $\varrho_{p(\cdot)}(u) >0$ or $p_2<\infty$ then
    \[
    	\min\{\varrho_{p(\cdot)}(u)^{1/p_1},\varrho_{p(\cdot)}
    	(u) ^{1/p_2}\}\leq  
    	\|u\|_{p(\cdot)} \le \max\{ \varrho_{p(\cdot)}(u)^{1/p_1},
    	\varrho_{p(\cdot)}(u) ^{1/p_2}\}
    \]
    for all $u\in\lp.$
   \end{lema}

\begin{teo}[H\"older's inequality]
    Let $p,q, s:\O\to[1,+\infty]$ be measurable functions such that
    \[
    	\frac1{p(x)}+\frac1{q(x)} = \frac1{s(x)}\quad\mbox{in }\O.
    \]
    Then
    \[
   		 \|f g\|_{s(\cdot)} \le 2 \|f\|_{p(\cdot)}\|g\|_{q(\cdot)},
    \]
    for all $f\in \lp$ and $g\in\lq $
\end{teo}

\medskip

Let $W^{1,p(\cdot)}(\Omega)$ denote the space of measurable 
functions $u$ such that, $u$ and the distributional derivative 
$\nabla u$ are in $L^{p(\cdot)}(\Omega)$. The norm
\[
	\|u\|_{{1,p(\cdot)},\O}\de \|u\|_{p(\cdot),\Omega} + 
	\| \nabla u \|_{p(\cdot),\Omega}
\]
makes $W^{1,p(\cdot)}(\Omega)$ a Banach space. 

We note
\[
	|u|_{{1,p(\cdot)},\O}\de \|\nabla u\|_{p(\cdot),\Omega}
\]
and we just write 
$\|u\|_{{1,p(\cdot)}}$ instead of $\|u\|_{{1,p(\cdot)},\O}$ and $|u|_{{1,p(\cdot)}}$ instead of $|u|_{{1,p(\cdot)},\O}$ 
when no confusion arises.

\begin{teo}\label{ref}
Let $p,p':\O\to[1,+\infty]$ be  measurable functions
such that 
\[
	\frac{1}{p(x)}+\frac{1}{p'(x)}=1\quad\mbox{ in }\O.
\] 
Then $L^{p'(\cdot)}(\Omega)$ is the dual of $L^{p(\cdot)}(\Omega)$.
Moreover, if $p_1>1$, $L^{p(\cdot)}(\Omega)$ and 
$W^{1,p(\cdot)}(\Omega)$ are reflexive.
\end{teo}

\medskip

We define the space $W_0^{1,p(\cdot)}(\Omega)$ as the closure of the
$C_0^{\infty}(\Omega)$ in $W^{1,p(\cdot)}(\Omega)$.
Then we have the following version of Poincar\'e inequity
(see Theorem 3.10 in \cite{KR}).

\begin{lema}[Poincar\'e inequity]
\label{poinc} If $p:\O\to[1,+\infty)$ is continuous in 
$\overline{\Omega}$, there exists a constant $C$ such that 
\[
	\|u\|_{{p(\cdot)}}\leq C\|\nabla
	u\|_{p(\cdot)}
\]
for all  $u\in W_0^{1,p(\cdot)}(\Omega).$
\end{lema}

\medskip

In order to have better properties of these spaces, we need more hypotheses on the regularity of $p(x)$.

We say that $p$ is \emph{$\log$-H\"{o}lder continuous} in $\O$ if there exists a
constant $C_{log}$ such that
\[
	|p(x) - p(y)| \leq \frac{C_{log}}{\log\,
	\left(e+\frac{1}{|x - y|}\right)}\quad \forall\, x,y\in\O.
\]

It was proved in  \cite{D}, Theorem 3.7, that if one assumes that 
$p$ is log--H\"{o}lder continuous then  $C^{\infty}(\bar{\Omega})$ 
is dense in $W^{1,p(\cdot)}(\Omega),$ see also
\cite{Di,DHHR, DHN, KR, Sam1}.

\medskip

\begin{prop}\label{epapa}
Let $p:\O\to[1,\infty)$ be a bounded 
$\log$-H\"{o}lder continuous function. Let $\beta>0,$
$D\subset \O$ and $h=\mbox{diam}(D).$ Then there exist constants $C$ independent of $h$ such that
\begin{equation}\label{eepapa}
h^{\beta(p(x)-p(y))}\le C \quad \forall x,y\in D.
\end{equation}
Moreover, if $p(x)$ is continuous in $\overline{D}$  then the  inequality \eqref{eepapa} holds  for all $x,y\in \overline{D}$.
\end{prop}

\medskip
We now state the Sobolev embedding theorem  (for the proofs see
\cite{DHHR}). Let,
\[
p^*(x):=
\begin{cases}
\frac{p(x)N}{N-p(x)} & \mbox{if } p(x)<N,\\
+\infty & \mbox{if } p(x)\ge N,
\end{cases}
\]
be the Sobolev critical exponent. Then we have the following theorem.
\begin{teo}\label{embed}
Let $\Omega$ be a Lipschitz domain and $p:\Omega\to [1,\infty)$ 
be a
$\log$--H\"{o}lder continuous function. Then the embedding 
$\wp\hookrightarrow \lpe$ is continuous.
\end{teo}


\subsection{The weak solution of  \texorpdfstring{\eqref{problema}}{(1.1)}}
The following results  can be found in \cite{DM}.

\begin{lema}\label{lemaux2}
Let $p\colon\O\to(1,+\infty)$ be a $\log$--H\"{o}lder continuous 
function, $f\in L^{q(x)}(\O)$ with $q'(x)\le p^*(x)$, $g\in \wp,$  
and $u$ be the weak solution  of \eqref{problema}. Then
 \[
 \|\nabla u\|_{p(\cdot)}\leq C
 \]
where $C$ is a constant depending on 
$\|f\|_{{q(\cdot)}},\|g\|_{1,p(\cdot)}$.
\end{lema}


\begin{teo}\label{H2convexa}
Let $\O$ be a bounded domain in $\R^2$ with convex boundary, 
$p\in\Lipc$ with $1<p_1\le p(x)\le2$,
$f\in L^{q(x)}(\O)$ with $q(x)\ge q_1> 2,$ and
 $g\in H^2(\O)$. Then the weak solution
of \eqref{problema} belongs to $H^2(\O).$
\end{teo}

\begin{remark}\label{continuidad}
If $\O$ is a bounded domain with Lipschitz boundary, we have
that $H^2(\O)$ is continuously imbedded in 
$C^{0,\alpha}(\overline{\O})$ for any $0\le\alpha<1,$ 
see Theorem 7.26 in \cite{GT}. 
Therefore, with this additional assumption, 
the weak solution of \eqref{problema} also belongs
to  $C(\overline{\O}).$
\end{remark}


\begin{remark}\label{aclaracion} 
The proof of Theorem \ref{H2convexa} follows using that there
exists $\{u_n\}_{n\in\mathbb{N}}\subset H^2(\O)$ such that
\[
\|u_n\|_{2,2}\le C=C(p(\cdot),\|f\|_{q(\cdot)},\|g\|_{2,2}) \quad 
\forall n\in\mathbb{N},
\]
and 
\[
u_n\rightharpoonup u\quad \mbox{weakly in } H^2(\O)
\]
where $u$ is the weak solution of \eqref{problema}. Therefore,
\[
\|u\|_{2,2}\le C=C(p(\cdot),\|f\|_{q(\cdot)},\|g\|_{2,2}). 
\]
See the proofs of Theorem 1.1 and Theorem  1.2 in \cite{DM}.
\end{remark}

\subsection{Finite Element Spaces}
Let $\Omega$ be a bounded convex domain in $\R^2$ with Lipschitz
boundary. Let $\O^h$ be a polygonal approximation to $\O$ defined
by $\overline{\O^h}=\bigcup_{\kappa\in\T^h}\kappa$ where 
$\T^h$ is a partitioning of $\O^h$ into a finite number of disjoint
open regular triangles $\kappa,$ each of maximum diameter bounded 
above by $h.$ In addition, for any two distinct triangles, their 
closures are either disjoint, or have 
a common vertex, or a common side. 
We also assume that $\O^h \subset\O,$ and if a vertex
belongs to $\partial\O^h$ then it also belongs to $\partial\O.$ 
 
Let
\[
	S^h\de\{v\in C(\overline{\O^h}): v|_{\kappa} 
	\mbox{ is linear } \forall \kappa \in T^h\},
\]
and $\pi_h\colon C(\overline{\O^h})\to S^h$ denote 
the interpolation operator such that for any 
$v\in C(\overline{\O^h}),$ $\pi_h v$ 
satisfies 
\[
\pi_hv(P)= v(P)
\] 
for all vertex $P$ associated to $\T^h.$

\medskip

The finite element approximation of \eqref{weaksol} is: Find $u^h\in 
S^h_g$ such that
\begin{equation}\label{fem1}
	\int_{\O^h}|\nabla u^h|^{p-2} \nabla 
	u^h \nabla v \, dx =\int_{\O^h} f v \,dx\quad \forall 
	v\in S_0^h
\end{equation}
where
\[
	S^h_g\de\{v\in S^h\colon v=g^h \mbox{ on } \partial \O^h\},
\]
and $g^h=\pi_h u$ with $u$ the solution of \eqref{weaksol}.

Observe that $\pi_h u$ is well defined due to 
$u\in C(\overline{\O}),$ see Remark \ref{continuidad}.

\begin{lema}\label{lemaux2h}
Let $f\in L^{q(x)}(\O)$ with $q'(x)\le p^*(x),$  $g\in \wp,$  and 
$u$ be the solution of \eqref{fem1}.  Then
\begin{equation}\label{cotagradienteuh}
 	\|\nabla u^h\|_{p(\cdot),\O^h}\leq C
\end{equation}
 where $C$ is a constant depending on 
 $\|f\|_{q(\cdot),\O}$ and $\|g^h\|_{1,p(\cdot),\O}$.
\end{lema}
\begin{proof}
The proof follows as in Lemma 4.1 of \cite{DM}, changing $u$ by $u^h$ and $g$ by $g^h$.
\end{proof}

The following interpolation theorem can by found in \cite{Ci}.
\begin{teo}\label{interpolacion}
For $m=0,1$ and for all $q\in [1,\infty]$ we have that,
\[
|v-\pi_h v|_{m,q,\O^h}\leq C h^{2-m} \|\nabla v\|_{q,\O} 
\]
for all $v\in W^{2,q}(\O),$ where
\[
|v-\pi_h v|_{m,q,\O^h}\de
\begin{cases}
	\|v-\pi_h v\|_{q,\O^h} & \mbox{if } m=0,\\
	\|\nabla(v-\pi_h v)\|_{q,\O^h} & \mbox{if } m=1.\\ 
\end{cases}
\]
\end{teo}

\subsection{Decomposition--Coordination method}\label{DeCor} 
Let $V,H$ be topological vectors spaces, $B\in\mathcal{L}(V,H)$
and $F\colon H\to\overline{\R},$ $G\colon V\to\overline{\R}$
be convex proper, lower semicontinuous functionals. To approximate 
the solution of variational problems of the following kind 
\begin{equation}\label{P}
	\min_{v\in V} F(Bv)+G(v)
\end{equation}
we use the following algorithm:

Given $r>0$ and
\[
	\{\eta_0,\lambda_1\}\in H\times H;
\]
then, $\{\eta_{n-1},\lambda_n\}$ known, we define
$\{u_n,\eta_n,\lambda_{n+1}\}\in V\times H\times H$ by
\[
	G(v)-G(u_n)+\langle \lambda_n, B(v-u_n)\rangle_{H}
	+r\langle Bu_n-\eta_{n-1}, B(v-u_n) \rangle_H\ge0
\]
for all $v\in V$;
\[
	F(\eta)-F(\eta_n)-\langle\lambda_n,\eta-\eta_n\rangle_H
	+r\langle\eta_n-Bu_n,\eta-\eta_n \rangle_H\ge 0
\] 
for all $\eta\in H;$
\[
	\lambda_{n+1}=\lambda_n+\rho_n(Bu_n-\eta_n)
\]
where $\rho_n>0.$

\medskip

The following theorem can be found in \cite{G}.

\begin{teo}\label{teo52}
Assume that $V$ and $H$ are finite dimensional and that
\eqref{P} has a solution $u.$ If 
\begin{itemize}
	\item $B$ is an injection;
	\item $G$ is convex, proper and lower semicontinous functional;
	\item $F=F_0+F_1$ with $F_1$ convex, proper and 
	lower semicontinous functional over $H$ and $F_0$ strictly convex
	and $C^1$ over $H;$
	\item $0<\rho_n=\rho<\dfrac{1+\sqrt{5}}2,$
\end{itemize}
then
\begin{align*}
	  u_n\to u &\qquad \mbox{strongly in } V,\\
	 \eta_n \to Bu &\qquad\mbox{strongly in } H,\\
	\lambda_{n+1}-\lambda_n\to0 &\qquad\mbox{strongly in } H,
\end{align*}
and $\lambda_n$  is bounded in  $H.$
 
\end{teo}

For more details about the Decomposition--Coordination method, we 
refer the reader to \cite{G} and references therein.

\section{Proofs of Theorem   \texorpdfstring{ {\ref{orden1}}}{1.1} and Theorem \texorpdfstring{ {\ref{c2alphaloc}}}{1.2} }{\label{orden p1/2}

In the remainder of this work we use the notation $0^0=1.$

\medskip

Let $1<p_1\le p(x)\le p_2<\infty$ and
$\sigma(x)\ge0,$ we define for any $v\in W^{1,p(\cdot)}(\T^h)$
\[
	\|v\|_{(p(\cdot),\sigma(\cdot))}
	\de\|(|\nabla u|+|\nabla v|)^{\frac{p(\cdot)
	-\sigma(\cdot)}{\sigma(\cdot)}}
	|\nabla v|\|_{\sigma(\cdot),\O^h},
\]
and
\[
	|v|_{(p(\cdot),\sigma(\cdot))}\de\int_{\Omega^h}(|\nabla u|
	+|\nabla v|)^{p(x)-\sigma(x)}|\nabla v|^{\sigma(x)} \, dx,
\]
where $u$ is the solution of \eqref{fem1}.

Observe that when $\sigma$ is constant we have  
$\|v\|^{\sigma}_{(p(\cdot),\sigma)}=|v|_{(p(\cdot),\sigma)}$.

\medskip

Before proving Theorem \ref{orden1}, we need some technical 
lemmas.
\begin{lema}\label{cotanormas}
Let $p,\sigma\colon\Omega\to(1,+\infty)$ be  measurable 
functions such that 
\[ 
	1<p_1\leq p(x)\leq \sigma(x)\leq\sigma_2<+\infty.
\] 
Then
\begin{equation}\label{cota3}
	\|v\|_{(p(\cdot),\sigma(\cdot))} \leq 
	\| |\nabla v|^{\nicefrac{p(\cdot)}{\sigma(\cdot)}}\|_{\sigma(\cdot),\O^h}.\end{equation}

Moreover, if there exits a constant $M$ such that
\begin{equation}\label{cota 1} 
	\varrho_{p(\cdot),\O^h}(|\nabla u|+|\nabla v|)\leq M
\end{equation}
then
\begin{equation}\label{cota2}
	\|\nabla v\|_{p(\cdot),\O^h}\leq C 
	\max{\{M^{1/\alpha_1}, M^{1/\alpha_2}\}} 
	\|v\|_{(p(\cdot),\sigma(\cdot))}
\end{equation}
where 
\[
	\alpha_1= \essinf_{x\in\O^h}
	\frac{\sigma(x)p(x)}{\sigma(x)-p(x)}
 	\quad\mbox{ and }\quad
	\alpha_2= \esssup_{x\in\O^h}
	\frac{\sigma(x)p(x)}{\sigma(x)-p(x)}.
\]
\end{lema} 

\begin{proof}
If $\sigma(x)\equiv p(x)$ a.e. then both inequalities are trivial. 

\medskip

Then, we will assume that 
$\esssup\{\sigma(x)-p(x)\colon x\in\O^h\}>0$.
Therefore, the inequality \eqref{cota3} holds due to 
$|\nabla u|+|\nabla v|\geq |\nabla v|$.

\medskip

To  prove inequality \eqref{cota2}, we will assume that  $|\nabla u|+|\nabla v|>0$ in a set of positive measure; the other case is trivial.

Let $w\colon\O^h \to \R,$ 
$w(x)\de(|\nabla u(x)|+|\nabla v(x)|)^{p(x)-\sigma(x)}.$
Then, by H\"older's inequality, we have
\begin{equation}\label{wsigma}
	\begin{aligned}
		\|\nabla v\|_{p(\cdot),\O^h}
		&=\|w^{-\nicefrac1{\sigma(\cdot)}} 
		w^{\nicefrac1{\sigma(\cdot)}} 
		|\nabla v|\|_{p(\cdot),\O^h}\\ 
		&\leq C \|w^{-\nicefrac1{\sigma(\cdot)}}
		\|_{\alpha(\cdot),\O^h}
		\|w^{\nicefrac1{\sigma(\cdot)}}
		|\nabla v|\|_{\sigma(\cdot),\O^h}\\
		&= C \|w^{-\nicefrac1{\sigma(\cdot)}}
		\|_{\alpha(\cdot),\O^h} 
		\|v\|_{(p(\cdot),\sigma(\cdot))},
\end{aligned} 
\end{equation}
where
$\alpha(x)\de\frac{\sigma(x) p(x)}{\sigma(x)-p(x)}.$ 
Observe that $\alpha(x)=\infty$ if only if $\sigma(x)=p(x)$. 

\medskip

On the other hand, 
by the definition of $\varrho_{\alpha(\cdot),\O^h}$  and 
\eqref{cota 1}, we get
\begin{align*}
	\varrho_{\alpha(\cdot),\O^h}
	\left(w^{\frac{-1}{\sigma}}\right)
	&=\varrho_{\alpha(\cdot),\O^h}
	\left(w^{\frac{-1}{\sigma}} \chi_{\{p\neq \sigma\}}\right)+
	\varrho_{\alpha(\cdot),\O^h}
	\left(w^{\frac{-1}{\sigma}}\chi_{\{p= \sigma\}}\right)\\
	&=\varrho_{p(\cdot),\O^h}
	\left((|\nabla u|+|\nabla v|)\chi_{\{p\neq \sigma\}}\right)+
	\varrho_{\alpha(\cdot),\O^h}\left(\chi_{\{\alpha= \infty\}}
	\right)\\
	&=\varrho_{p(\cdot),\O^h}
	\left((|\nabla u|+|\nabla v|)\chi_{\{p\neq \sigma
	\}}\right)\\
	&\leq \varrho_{p(\cdot),\O^h}(|\nabla u|+|\nabla v|)\\
	&\leq M.
\end{align*}

Finally, let $\displaystyle\alpha_1=\essinf_{x\in\O^h}\alpha(x)$ and $\displaystyle\alpha_2=\esssup_{x\in\O^h}\alpha(x).$ 
Observe that $\alpha_1<\infty$ due to 
$\esssup\{\sigma(x)-p(x)\colon x\in\O^h\}>0.$ 
Therefore, by Lemma \ref{compinf}, we have that
\[
\|w^{\frac{-1}{\sigma}}\|_{\alpha(\cdot),\O^h}
\leq \max\{M^{\nicefrac{1}{\alpha_1}}, 
M^{\nicefrac{1}{\alpha_2}}\}.
\]
Combining this inequality with \eqref{wsigma} we obtain \eqref{cota2}.
\end{proof}

\begin{remark}
Let $u$ and $u^h$ be the unique solutions of \eqref{weaksol} and 
\eqref{fem1}, respectively.
Then
\begin{align*}
	J_{\O}(u)\le& J_{\O}(v) \quad \forall 
	v\in W^{1,p(\cdot)}_g(\O),\\
	J_{\O^h}(u^h)\le& J_{\O_h}(v) \quad \forall v\in S^h_g,
\end{align*}
where
\[
	J_\Lambda (v)\de\int_\Lambda \frac{1}{p(x)}|\nabla v|^{p(x)}
	\, dx -\int_\Lambda f v \, dx
\]
with $\Lambda=\O$ or $\Lambda=\O^h.$

Observe that $J_\Lambda$ is G\^axteaux differentiable with
\[
	J_\Lambda^\prime(u)(v)=
	\int_{\Lambda} |\nabla u|^{p(x)-2}\nabla u \nabla v \, dx 
	- \int_{\Lambda} f v \, dx.
\]
for any $v\in W^{1,p(\cdot)}(\Lambda).$
\end{remark}

\begin{lema}\label{cotaseminorma}
Let $p\colon\O\to(1,2)$ be a log--H\"older continuous 
function. Let $u$ and $u^h$ be the solutions of \eqref{weaksol} 
 and \eqref{fem1}, respectively. Then, for any 
 $\delta_1,\delta_2\colon\O\to [0,+\infty)$ be measurable
 functions such that   
 $0\leq \delta_1(x)\leq \delta^+ < 2,$ we have 
\begin{equation*}
	|u-u^h|_{(p(\cdot),2+\delta_2(\cdot))}\leq C |u-v|_{(p(\cdot)
	,	2-\delta_1(\cdot))}
\end{equation*}
for all $v\in S_g^h.$
\end{lema}
\begin{proof} 
We first observe that for all $v\in S^h_g$
\begin{equation}\label{derivadaJ}
	J_{\O^h}(v)-J_{\O^h}(u)=A(v)+J^{\prime}_{\O^h}(u)(v-u),
\end{equation}
where
\[
	A(v)=\!\int_0^1\!\!\int_{\O^h} 
	\left(|\nabla (u+sw)|^{p(x)-2} \nabla(u+sw)
	-|\nabla u|^{p(x)-2} \nabla u\right) \nabla w \, dx \, ds,
\]
with $w=v-u.$

Observe that, for all $v_1, v_2.$ and $s\in [ 0, 1]$ we have
\begin{equation}\label{pequeñacota}
\frac{s}{2} (|\nabla v_1|+|\nabla v_2 |)\leq |\nabla(v_1+sv_2)|+|\nabla v_1|\leq 2(|\nabla v_1|+|\nabla v_2 |).
\end{equation}

By \eqref{21a} and \eqref{pequeñacota}, for $q_1(x)=1-\delta_1(x)$
and $q_2(x)=p(x)-2-\delta_1(x)$ we have
\begin{equation}\label{cota11}	
\begin{aligned}
	|A(v)|\leq& C \int_0^1\!\!\! \int_{\O^h}\!\!\! 
					\left(|\nabla (u\!+\!sw)|
					+|\nabla u|\right)^{q_2(x)}\!\! 
					|\nabla w|^{1+q_1(x)} 
					s^{q_1(x)} dx  ds\\
			\leq& C\!\!   \int_{\O^h}\!\!\!
					\left(|\nabla w|+|\nabla u|\right)^{q_2(x)} 
					|\nabla w|^{1-q_1(x)}\!\!
					\left( \int_0^1 s^{q_1(x)} ds 
					\right) dx\\
			\leq& \frac{C}{2-\delta^+} \int_{\O^h} 
			 		\left(|\nabla w|
			 		+|\nabla u|\right)^{p(x)-2-\delta_1(x)} 
			 		|\nabla w|^{2-\delta_1(x)} dx\\
		=&C|w|_{(p(\cdot),2-\delta_1(\cdot))} \\ 
		=&C|u-v|_{(p(\cdot),2-\delta_1(\cdot))}.
\end{aligned}
\end{equation}

\medskip

On the other hand, by \eqref{21b} and \eqref{pequeñacota},
for $q_3(x)=1+\delta_2(x)$ and $q_4(x)=p(x)-2-\delta_2(x)$ we have
\begin{equation}\label{cota12}
\begin{aligned}
	|A(v)|\geq &C \!\!\int_0^1 \!\!\!\int_{\O^h}\!\!\!
				\left(|\nabla (u+sw)|+|\nabla u|\right)^{q_4(x)} 
				|\nabla w|^{1+q_3(x)} s^{q_2(x)} dx  ds\\
			\geq &C \!\!\int_{\O^h}\!\!\! 
				\left(|\nabla w|+|\nabla u|\right)^{q_4(x)} 
				|\nabla w|^{1+q_3(x)}\left( \int_0^1  s^{p(x)-1} ds
				\right) dx\\
			\geq &\frac{C}{p_2}\int_{\O^h}
				\left(|\nabla w|+|\nabla u|\right)^{p(x)-2-
				\delta_2(x)} |\nabla w|^{2+\delta_2(x)} dx\\
			=& C|w|_{(p(\cdot),2+\delta_2(\cdot))}\\
			=& C|u-v|_{(p(\cdot),2+\delta_2(\cdot))} 
\end{aligned}
\end{equation}
for all $v\in S^h_g.$

Using \eqref{derivadaJ},  we have that
\[
	A(u^h)+J^{\prime}_{\O^h}(u)(u^h-u)\leq A(v)
	+J^{\prime}_{\O^h}(u)(v-u) \quad\forall
	v\in S^h_g
\]
due to $u^h$ is a minimizer of $J_{\O^h}.$ Then,
\[
	A(u^h)\leq A(v)+J^{\prime}_{\O^h}(u)(v-u^h)\quad\forall
	v\in S^h_g.
\]
Therefore, by \eqref{cota11} and \eqref{cota12}, we have
\[
	|u-u^h|_{(p(\cdot),2+\delta_2(\cdot))}\leq 
	C |u-v|_{(p(\cdot),2-\delta_1(\cdot))}
	+|J^{\prime}_{\O^h}(u)(v-u)|\quad\forall
	v\in S^h_g.
\]

Finally, for any $v\in S^h_g,$ 
since $\O^h$ is Lipschitz, $\O^h\subset \O$  and 
$\varphi=v-u^h \in S^h_0$, we can extend $\varphi$ to be zeros in 
$\O\setminus\O^h$, by a function $\hat{\varphi}\in \wpc.$ 
Then
\[
	J^{\prime} _{\O^h}(u)(\varphi)=J^{\prime}_{\O}(u)
	(\hat{\varphi})=0
\]
due to $u$ is a minimizer of $J_{\O}.$ Therefore 
$J^{\prime} _{\O^h}(u)(v-u^h)=0$ for all $v\in S^h_g.$ 
This completes the proof. 
\end{proof}

Now we are able to prove Theorem \ref{orden1}.

\begin{proof}[Proof of Theorem \ref{orden1}]
We begin by noting that, by Lemma \ref{lemaux2}, 
Lemma \ref{lemaux2h}, and \eqref{cotagradienteuh}, 
we can apply  Lemma \ref{cotanormas} with $\sigma=2.$ 
We get
\[
	|u-u^h|^2_{1,p(\cdot),\O_h}\leq C \|u-u^h\|^2_{(p(\cdot),2)}= 
	C |u-u^h|_{(p(\cdot),2)}.
\]
Then, taking $\delta_1(x)=2-p(x)$ and $\delta_2(x)\equiv0$  in Lemma \ref{cotaseminorma}, we 
have that 
\[
	|u-u^h|^2_{1,p(\cdot),\O_h}\leq
	 C |u-v|_{(p(\cdot),p(\cdot))}=C \rho_{p(\cdot),\O_h} 
	 (|\nabla u-\nabla v|)\quad \forall v \in S_g^h.
\]
By Lemma \ref{compinf}, we have that
\begin{equation}\label{cea}
	|u-u^h|_{1,p(\cdot),\O_h}\leq C \max\left
	\{|u-v|_{1,p(\cdot),\O_h}^{\nicefrac{p_1}{2}},
	|u-v|_{1,p(\cdot),\O_h}^{\nicefrac{p_2}{2}}
	\right\}{\quad \forall v\in S_g^h}.
\end{equation}

On the other hand, by Poincar\'e inequality and triangle inequality, 
\begin{equation}\label{piinter}
\begin{aligned}
	\|u-u^h\|_{1,p(\cdot),\Omega^h}&\leq 
				\|u-\pi_h u\|_{1,p(\cdot),\Omega^h}+ 
				\|u^h-\pi_h u\|_{1,p(\cdot),\Omega^h}\\
				&\leq C\left( \|u-\pi_h u\|_{1,p(\cdot),\Omega^h}+
				|u^h-\pi_h u|_{1,p(\cdot),\Omega^h}\right)\\
				&\leq C\left(\|u-\pi_h u\|_{1,p(\cdot),\Omega^h}+
				|u^h-u|_{1,p(\cdot),\Omega^h}\right).
\end{aligned}
\end{equation}
Using Theorem \ref{interpolacion} for $m=0,1$ and $q=p_2$,
Theorem \ref{H2convexa} and, Remark \ref{aclaracion}, we have that 
\begin{equation}\label{interp2}
	|u-\pi_h u|_{m,p(\cdot),\O^h}
	\leq C |u-\pi_h u|_{m,p_2,\O^h}\leq C h^{2-m} |u|_{2,p_2,\O}.
\end{equation}

Taking $v=\pi_h u$ in \eqref{cea} and, using  \eqref{piinter} and \eqref{interp2}, we get
\[
	\|u-u^h\|_{1,p(\cdot),\Omega^h}
	\leq C (h|u|_{2,p_2,\O}+ (|u|_{2,p_2,\O}h)^{p_1/2}),
	\quad \mbox{ if } h |u|_{2,p_2,\O}\leq 1. 
\]
Finally, using Remark \ref{aclaracion} and that $p_2\leq 2,$ 
we obtain the desired result.
\end{proof}

\medskip

Lastly, we prove Theorem \ref{c2alphaloc}.

\begin{proof}[Proof of Theorem \ref{c2alphaloc}]
By Lemma \ref{cotanormas} with $\sigma=2$ and taking 
$\delta_1(x)\!=\!\delta_2(x)\equiv0$ in Lemma \ref{cotaseminorma}, 
we obtain
\begin{align*}
	|u-u^h|^2_{1,p(\cdot),\O^h}&\leq C |u-u^h|_{(p(\cdot),2)}\\
	&\leq C |u-\pi_h u|_{(p(\cdot),2)}\\ 
	&=C \sum_{\tau\in \T_h} \int_{\tau} \left(|\nabla u|+
	|\nabla (u-\pi_h u)|\right)^{p(x)-2} 
	|\nabla (u-\pi_h u)|^2\,dx\\
	& =:I. 
\end{align*}

On the other hand, by interpolation inequality, we have
\begin{equation}\label{inttau}
	|\nabla (u-\pi_h u)(x)|\leq C h \|H[u]\|_{L^{\infty}(\tau)}
	\leq C H[u](x)+ C h^{1+\alpha^+}\quad \forall x \in \tau,
\end{equation}
due to $u\in C^{2,\alpha^+}(\tau).$ 

We also have $q(t)=(a+t)^{p-2} t^2$ with $a>0$ is increasing and
hence $q(|t_1+t_2|)\leq 2(q(|t_1|)+q(|t_2|)).$ Then, by \eqref{inttau}
and since $p(x)\leq 2,$ we get

\begin{align*}
	I&\leq C \sum_{\tau\in \T_h}  
	h^2\int_{\tau}\left(|\nabla u|+Ch H[u]\right)^{p(x)-2} 
	H[u]^2\, dx \\&
	+\sum_{\tau\in \T_h} \int_{\tau} 
	\left(|\nabla u|+Ch^{1+\alpha^+}\right)^{p(x)-2}
	h^{2(1+\alpha^+)}\, dx\\
	&\leq C  h^2  \int_{\O^h} |\nabla u|^{p(x)-2} H[u]^2\, dx 
	+C \sum_{\tau\in \T_h} \int_{\tau} h^{p(x)(1+\alpha^+)}\, dx\\
	&\leq C  h^2  \int_{\O^h} |\nabla u|^{p(x)-2} H[u]^2\, dx +C h^{2}
\end{align*}
where in the last inequality we are using Proposition \ref{epapa}.
\end{proof}

\begin{remark}
Since,
\[
	\int_\O |\nabla u|^{p(x)-2}{H[u]^2}\,dx\leq \int_\O 
	|\nabla u|^{p_2-2}{H[u]^2}\,dx
	+\int_\O |\nabla u|^{p_1-2}{H[u]^2}\,dx
\]
we have, by Lemma 3.1 in \cite{LB2}, that \eqref{int11} holds 
if $u\in W^{3,1}(\O).$
\end{remark}

\begin{remark}
We can see that \eqref{reghold} can be interpreted as follows: 
in order to have optimal rate of convergence  we only need
$C^2$ regularity of the solution, in regions  where the maximum of $p(x)$ is $2$, and 
we need, for example, $C^{2,1}$ regularity of the solution, only in regions  where the 
function $p(x)$ is near $1$.
\end{remark}

The next example is a generalization of \cite[Example 3.1]{LB}.

\begin{example}\label{ejemplo1}
We consider the radially symmetric version of the problem. Let 
$\O=B_1(0)$,$f(x)=F(r),$ $p(x)=P(r)$ and $g$ is constant, where  
$r=|x|.$ We assume that
\begin{equation}\label{conejemplo1}
	P(r)\neq 2 \quad \mbox{ if } \quad 
	\frac{1}{r}\int_0^r t F(t)\, dt=0,
\end{equation}
and  for each  $\tau \in \T^h$
\begin{equation}\label{conejemplo2}
 	p \in C^{1,\beta}(\tau), f \in C^{\beta}(\tau) 
 	\mbox{ with }\beta\geq \alpha^+ .
\end{equation}
 We will  see that \eqref{int11} and \eqref{reghold} of Theorem \ref{c2alphaloc} hold. 
 
We first observe that
\[
	u(x)=U(r)=-\int_r^1 Z(t) |Z(t)|^{\frac{2-P(t)}{P(t)-1}} \, dt+g
\]
where
\[
	Z(r)=(|U^\prime|^{P-2} U^\prime)(r)=
	-\frac{1}{r} \int_0^r t F(t)\, dt.
\]

If we derive $Z$, using  that $|Z|=|U|^{P-1},$ we have that
\[
	U^{\prime\prime}=\frac{1}{P-1} Z^\prime 
	|Z|^{\frac{2-P}{P-1}}-\frac{1}{(P-1)^2}|Z|^{\frac{2-P}{P-1}} 
	P^\prime\log(|Z|) Z.
\]
Observe that $U^{\prime\prime}$ is 
well define since  \eqref{conejemplo1}  implies 
\begin{equation}\label{cond223}
	Z(r) \neq 0 \mbox{ if } P(r)=2.
\end{equation}

On the other hand, by \eqref{conejemplo2}, we have that  
\begin{equation}\label{unauna2}
 	P \in C^{1,\beta}(a,b) \mbox{ and } F\in C^{\beta}(a,b) 
\end{equation}
where 
$a=\min\{|x|: x\in \partial \tau\}$ and 
$b=\max\{|x|: x\in \partial \tau\}.$

Then
\begin{equation}\label{unauna}
	Z\in C^{1,\beta}(a,b).
\end{equation}
and therefore
\begin{equation}\label{una}
	|Z|^{\frac{2-P}{P-1}}\in C^{\frac{2-P^+}{P^+-1}}(a,b)
\end{equation}
where $P^+=\displaystyle\max_{r\in[a,b]} P(r)$.
 
On the other hand, since $\log(t) t$ is H\"older continuous for any 
exponent, we have that
\begin{equation}\label{dos}
	|Z|^{\frac{2-P}{P-1}}\log(|Z|) Z \in 
	C^{\frac{2-P^+}{P^+-1}}(a,b),
\end{equation}
and then, by \eqref{unauna2}--\eqref{dos}, we have that 
\[
U^{\prime\prime}\in C^{\gamma}(a,b) \mbox{ where } 
\gamma=\min\left\{\beta, \frac{2-P^+}{P^+-1}\right\}.
\]
Finally, since $Z(0)=0$ and by \eqref{cond223}, we have that 
$U^\prime(0)=U^{\prime\prime}(0)=0$
so $u\in C^{1,\gamma}(\tau)$ and \eqref{reghold} holds.

If we define 
$\hat{H}[u]^2=(u_{x_1 x_1})^2+2(u_{x_1 x_2})^2+(u_{x_2 x_2})^2$ 
we have
\[
H[u]\le 3\hat H[u],
\]
\begin{equation}\label{Hugrad}
	\hat{H}[u]^2 |\nabla u|^{p-2}=(U^{\prime\prime})^2 
	|U^{\prime}|^{P-2}+\frac{|U^{\prime}|^p}{r^2}.
\end{equation}

First, since $P,Z\in C^1$ and by \eqref{cond223}, we have that
\begin{equation}\label{primter}
\begin{aligned}
	(U^{\prime\prime})^2 |U^{\prime}|^{P-2}
	&=\frac{1}{(P-1)^2}(Z^\prime)^2 |Z|^{\frac{2-P}{P-1}}-
	\frac{2}{(P-1)^3} |Z|^{\frac{2-P}{P-1}} P' \log(|Z|) Z\\
	&+\frac{2 (P^{\prime})^2}{(P-1)^4} |Z|^{\frac{2-P}{P-1}} 
	\log^2(|Z|) Z^2 \quad \in L^{\infty}(0,1).
\end{aligned}
\end{equation}
On the other hand using that  $Z(0)=0$ and $Z\in C^1$ we have that
\begin{equation}\label{segter}
\frac{|U'|^p}{r^2}=\frac{|Z|^{\frac{P}{P-1}}}{r}\in L^{\infty}(0,1).
\end{equation}

Therefore, by \eqref{Hugrad}--\eqref{segter}
\[
	\int_{\O}\hat{H}[u]^2 |\nabla u|^{p-2}\, dx=2 \pi 
	\int^1_0(U^{\prime})^2 |U^{\prime}|^{P-2}r+
	\frac{|U'|^p}{r}\, dr <\infty,
\]
so  \eqref{int11} holds.

\end{example}

\section{Numerical examples}\label{ejemplos}
In this section, for each $h\geq 0$ we approximate the solution $u^h$ of \eqref{fem1} by the sequence $u^h_n$ driven by the 
algorithm described in Subsection \ref{DeCor}. For simplicity we will denote $u^h_n=u_n$.

Let $V=S^h_g,$   
\[
		H=\left\{\eta: \R^2 \to \R^2: \eta|_{\kappa}=\mbox{ constant } \right\},
\]
\[
		F(\eta)=\int_{\Omega} \frac{|\eta|^{p(x)}}{p(x)}\, dx, \quad  G(v)=\int_{\Omega} f v\, dx,
\]
and $B:V\to H$ defined by $B(v)=\nabla v.$
Then 
\[
		J_{\Omega^h}(v)=F(B(v))+G(v).
\]

If we take  $\rho_n=r=1$ then the algorithm is: 

Given 
\[
\{\eta_0,\lambda_1\}\in H \times H ,
\]

 then, $\{\eta_{n-1},\lambda_n\}$ known, we define $\{u_{n},\eta_{n},\lambda_{n+1}\}\in V\times H \times H$ by
\begin{align}
	 &\label{M1}
	\int_{\Omega}\nabla u_n\nabla v\, dx= 	\int_{\Omega} fv \, dx+ 	\int_{\Omega} (\eta_{n-1}-\lambda_n) \nabla v\, dx, \quad & \forall v\in V,\\
	&\label{M2} \int_{\Omega}(|\eta_n|^{p(x)-2} \eta_n +\eta_n) \eta\, dx  =\int_{\Omega} (\lambda_n+\nabla u_n) \eta\, dx  & \forall \eta\in H,
\end{align}
\begin{equation*}
	\lambda_{n+1}=\lambda_n +(\nabla u_n- \eta_n).
\end{equation*}

\begin{remark}
Since $V,H,F,G,B,\rho_n$ and $r$  satisfy the assumptions of  Theorem \ref{teo52}  
then the conclusions of Theorem \ref{teo52} are satisfied, that is,
$u_n \to u^h$ and $\nabla u_n \to \nabla u^h$.
\end{remark}

Observe that \eqref{M1} can be replace by,
$$MU_n=F_n,$$
where 
\begin{align*}
M_{ij}&=\int_{\Omega} \nabla \varphi_i  \nabla \varphi_j \, dx,\\
F_{n,_j}&=\int_{\Omega}\varphi_j f \, dx+ \int_{\Omega} ( \eta_{n-1}-\lambda_n) \nabla \varphi_j\, dx,\end{align*}
and $\{\varphi_j\}_{j\leq N}$ is a basis of $V$ with $N=\mbox{dim}(V)$. Thus
$$u_n=\sum_{j=1}^N u_{n,j} \varphi_j.$$

On the other hand, we  define
$\eta_{n,{\kappa}}=\eta_n|_{\kappa}$, in the same way we define $\lambda_{n,{\kappa}}$ and $\nabla_{\kappa}u_n$. 
We can see from  \eqref{M2} that $\eta_{n,{\kappa}}$  satisfies

\[
\left(\dfrac{1}{|\kappa|}\int_{\kappa}|\eta_{n,{\kappa}}|^{p(x)-2}\,dx +1\right)\eta_{n,{\kappa}}=\lambda_{n,{\kappa}}+ \nabla_{\kappa}u_n.
\]

Let $\bar{p}_{\kappa}=p(\bar{x}_{\kappa})$, where $\bar{x}_{\kappa}$  is the varicenter of $\kappa$. Then using a 
quadrature rule for the first term, we can approximate $\eta_{n,{\kappa}}$ by the equation,
\[
(|\eta_{n,{\kappa}}|^{\bar{p}_{\kappa}-2}+1)\eta_{n,{\kappa}}=\lambda_{n,{\kappa}}+ \nabla_{\kappa}u_n, 
\]

thus $|\eta_{n,{\kappa}}|$ solves
\[
|\eta_{n,{\kappa}}|^{\bar{p}_{\kappa}-1}+|\eta_{n,{\kappa}}|=|\lambda_{n,{\kappa}}+\nabla_{\kappa}u_n|,
\]
and therefore
\[
\eta_{n,{\kappa}}=\frac{\lambda_{n,{\kappa}}+\nabla_{\kappa}u_n}{|\eta_{n,{\kappa}}|^{\bar{p}_{\kappa}-2}+1}.
\]

Summarizing, each iteration of the algorithm  can be reduce to the 
following:

\medskip

Find $\{u_{n},\eta_{n},\lambda_{n+1}\}\in V\times H \times H$ such that
\[
	 u_n=\sum_{j=1}^N U_{n,j} \varphi_j,
\]
where $U_n$ solves,
\begin{equation}\label{lineal}
MU_n=F_n;
\end{equation}
\[
	\eta_{n,{\kappa}}=\frac{\lambda_{n,{\kappa}}
	+\nabla_{\kappa}u_n}{b^{\bar{p}_{\kappa}-2}+1}
\]
where $b\in\R_{\geq 0}$ solves  
\begin{equation}\label{nolineal}b^{\bar{p}_{\kappa}-1}+b=|\lambda_{n,{\kappa}}
+\nabla_{\kappa}u_n|,\end{equation}and
\[
\lambda^{n+1}=\lambda_n +(\nabla u_n- \eta_n).
\]
Observe that each step of the algorithm consists in solving the linear equation \eqref{lineal} and then the one dimensional nonlinear equation \eqref{nolineal}.

\medskip

We now apply the algorithm to a family of examples. For each $h$, we 
use a stooping time criterion and we approximate $u_n$ by $u^h_n$, 
and finally  we compute $\|u^h_n-u\|_{W^{1,p(\cdot)}(\Omega)}$.

\medskip
 
In the following example, we have considered a rectangular domain 
$\Omega=[-1\ 1]\times [-1\ 1]$ and a uniform  mesh, with linear 
finite elements in all triangles. We  denote by $N$ the number 
of degrees of freedom in the finite element approximation. 


We consider the case  $f = 0,$ and the following   function $p(x)$, 

\[
	p(x)=
	\begin{cases} 
		1+\left({\dfrac{b}{2}(x_1+x_2)+1+b}\right)^{-1} 
		&\mbox{ if } b\neq 0,\\
		2 &\mbox{ if } b =0.
	\end{cases}
\]


It is easy to see that the solution of \eqref{problema} is
\[
	u(x)=
	\begin{cases} 
		\dfrac{\sqrt{2}e^{b+1}}b
		\left(e^{\frac{b}{2}(x_1+x_2)}-1\right) 
		&\mbox{ if } b \neq 0,\\[.5cm] 
		\dfrac{\sqrt{2}e}{2}  (x_1+x_2) &\mbox{ if } b =0.
	\end{cases}
\]

 
The experimental results for different values of $b$ and $N$ are 
shown in the following table, where  $e=u-u_n^h$. 

\begin{center}
	\begin{table}[H]
	\begin{tabular}{|c|c|c|c|c|c|c|c|}
    	\hline
     	\backslashbox{$b$}{$N^{\nicefrac12}$} & 20  & 40  & 60 
     	& 80 & 100 & 120 & 140 \\ \hline \hline
		0.1 &    0.0200  &  0.0100   & 0.0067   & 0.0050  
		&  0.0040  & 0.0033   & 0.0029  \\ \hline
		0.5 & 0.1707   & 0.0848  &  0.0567  &  0.0427  &  0.0342  
		& 0.0286  &  0.0245\\ \hline
		1 &      0.6704   & 0.3341   & 0.2244   & 0.1692  &  0.1357 
		& 0.1135  &  0.0973\\ \hline
		2. &      5.5457   & 2.7592  &  1.8683 &   1.3750  &  1.1055 
		& 0.9250  &  0.7940\\ \hline
		2.5 &     5.5457  &  2.7592  &  1.8683   & 1.3750  &  1.1055
		& 2.3770   & 2.0434\\ \hline
		3 &     14.2471   & 7.2017   & 4.8641  &  3.6136  &  2.8534 
		& 6.6850   & 5.8923\\ \hline
      \end{tabular}
      \vspace{0.25cm}  
	\caption{ $\|e\|_{1,p(\cdot)}$ respect to  $N^{1/2}$ and $
	b$}
\end{table}
\end{center}
 
Figure \ref{fig:figura1} exhibits a plot, for different values of 
$b$, of $\log(\|e\|_{1,p(\cdot)})$ respect to $N^{1/2}$. 

\begin{center}
\begin{figure}[H]
 \includegraphics[scale=.70]{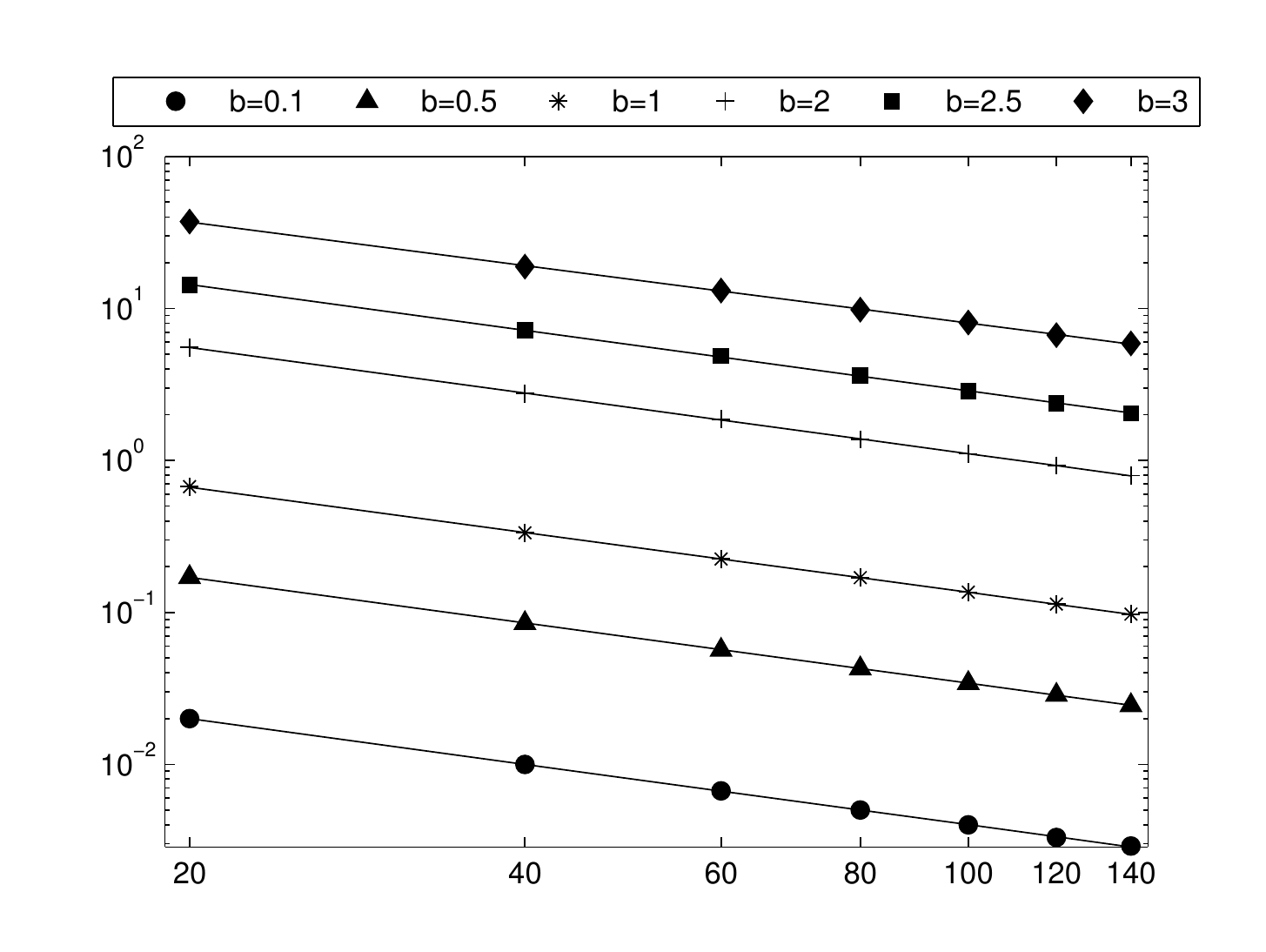}
  \caption{$\|e\|_{1,p(\cdot)}$ respect to $N^{1/2}$ in  loglog scale}
  \label{fig:figura1}
\end{figure}
\end{center}

  
Fitting these values by the model $\|e\|_{1,p(\cdot)}\sim C N^{\nicefrac{-\alpha}{2}}$ using least square approximation gives us the  
results of Table \ref{tabla22}.
\begin{center}
  \begin{table}[H]
		\begin{tabular}{|c|c|c|c|}
   		\hline
   			$b$& $p_1$& $\alpha$  & $C$ \\ \hline \hline
			0.1 & 1.83& 0.9984 &  0.1992\\ \hline
 			0.5  & 1.5& 0.9961  &  1.6842\\ \hline
  			1 & 1.33& 0.9900    & 6.52289\\ \hline
   			2 &1.2&  0.9998    & 55.3856\\ \hline
   			2.5& 1.16& 1.0007 &  143.9890\\ \hline
  			3  &  1.14&  0.9495 	&	329.2832 \\ \hline
		\end{tabular}
		 \vspace{0.25cm}          
		\caption{Numerical order}\label{tabla22}
	\end{table}
\end{center}

Observe that the numerical rate of convergence is still  of order one.

We also observe that $p_1$ is close to one when $b>>1$, for example 
$p_1=1.14$ if $b=3$. 
Table \ref{tabla22} shows that the constant $C$ increases 
when $p_1$ is near to one. In fact, the  bound of the 
$\|u\|_{H^2(\O)}$ and the constants $C_1$ and $C_2$ in 
Lemma \ref{desigualdades} depend on $\nicefrac{1}{(p_1-1)}$, see 
\cite{LB2,DM}.

\bibliographystyle{amsplain}
\bibliography{las}

\def\ocirc#1{\ifmmode\setbox0=\hbox{$#1$}\dimen0=\ht0 \advance\dimen0
  by1pt\rlap{\hbox to\wd0{\hss\raise\dimen0
  \hbox{\hskip.2em$\scriptscriptstyle\circ$}\hss}}#1\else {\accent"17 #1}\fi}
  \def\cprime{$'$} \def\ocirc#1{\ifmmode\setbox0=\hbox{$#1$}\dimen0=\ht0
  \advance\dimen0 by1pt\rlap{\hbox to\wd0{\hss\raise\dimen0
  \hbox{\hskip.2em$\scriptscriptstyle\circ$}\hss}}#1\else {\accent"17 #1}\fi}
\providecommand{\bysame}{\leavevmode\hbox to3em{\hrulefill}\thinspace}
\providecommand{\MR}{\relax\ifhmode\unskip\space\fi MR }
\providecommand{\MRhref}[2]{%
  \href{http://www.ams.org/mathscinet-getitem?mr=#1}{#2}
}
\providecommand{\href}[2]{#2}
\begin{thebibliography}{10}

\bibitem{LB2}
John~W. Barrett and W.~B. Liu, \emph{Finite element approximation of the
  {$p$}-{L}aplacian}, Math. Comp. \textbf{61} (1993), no.~204, 523--537.
  \MR{1192966 (94c:65129)}

\bibitem{BCE}
Erik~M. Bollt, Rick Chartrand, Selim Esedo{\=g}lu, Pete Schultz, and Kevin~R.
  Vixie, \emph{Graduated adaptive image denoising: local compromise between
  total variation and isotropic diffusion}, Adv. Comput. Math. \textbf{31}
  (2009), no.~1-3, 61--85.

\bibitem{CHP}
Erich Carelli, Jonas Haehnle, and Andreas Prohl, \emph{Convergence analysis for
  incompressible generalized {N}ewtonian fluid flows with nonstandard
  anisotropic growth conditions}, SIAM J. Numer. Anal. \textbf{48} (2010),
  no.~1, 164--190.

\bibitem{CLR}
Yunmei Chen, Stacey Levine, and Murali Rao, \emph{Variable exponent, linear
  growth functionals in image restoration}, SIAM J. Appl. Math. \textbf{66}
  (2006), no.~4, 1383--1406 (electronic).

\bibitem{Ci}
Ph. Ciarlet, \emph{The finite element method for elliptic problems}, vol.~68,
  North-Holland, Amsterdam, 1978.

\bibitem{DLM}
Leandro~M. Del~Pezzo, Ariel~L. Lombardi, and Sandra Mart{\'{\i}}nez,
  \emph{Interior penalty discontinuous {G}alerkin {FEM} for the
  {$p(x)$}-{L}aplacian}, SIAM J. Numer. Anal. \textbf{50} (2012), no.~5,
  2497--2521. \MR{3022228}

\bibitem{DM}
Leandro~M. Del~Pezzo and Sandra Mart{\'{\i}}nez, \emph{{$H^2$} regularity for
  the {$p(x)$}-{L}aplacian in two-dimensional convex domains}, J. Math. Anal.
  Appl. \textbf{410} (2014), no.~2, 939--952. \MR{3111880}

\bibitem{Di}
L.~Diening, \emph{Theoretical and numerical results for electrorheological
  fluids}, Ph.D. thesis, Institut f\"ur Angewandte Mathematik, Mathematische
  Fakult\"at, 7 2002.

\bibitem{D}
\bysame, \emph{Maximal function on generalized {L}ebesgue spaces
  {$L^{p(\cdot)}$}}, Math. Inequal. Appl. \textbf{7} (2004), no.~2, 245--253.
  \MR{2057643 (2005k:42048)}

\bibitem{DHHR}
Lars Diening, Petteri Harjulehto, Peter H{\"a}st{\"o}, and Michael
  R{\ocirc{u}}{\v{z}}i{\v{c}}ka, \emph{Lebesgue and {S}obolev spaces with
  variable exponents}, Lecture Notes in Mathematics, vol. 2017, Springer,
  Heidelberg, 2011. \MR{2790542}

\bibitem{DHN}
Lars Diening, Peter H{\"a}st{\"o}, and Ale{\v{s}} Nekvinda, \emph{Open problems
  in variable exponent lebesgue and sobolev spaces}, P. Dr{\'a}bek, J.
  R{\'a}kosn{\i}k, FSDONA04 Proceedings, Milovy, Czech Republic (2004), 38--58.

\bibitem{EL}
Carsten Ebmeyer and WB. Liu, \emph{Quasi-norm interpolation error estimates for
  the piecewise linear finite element approximation of {$p$}-{L}aplacian
  problems}, Numer. Math. \textbf{100} (2005), no.~2, 233--258.

\bibitem{GT}
D.~Gilbarg and N.~S. Trudinger, \emph{Elliptic partial differential equations
  of second order}, Grundlehren der Mathematischen Wissenschaften [Fundamental
  Principles of Mathematical Sciences], vol. 224, Springer-Verlag, Berlin,
  1983.

\bibitem{G}
Roland Glowinski, \emph{Numerical methods for nonlinear variational problems},
  Scientific Computation, Springer-Verlag, Berlin, 2008, Reprint of the 1984
  original. \MR{2423313 (2009c:65002)}

\bibitem{KR}
Kov\'a\v{c}ik and R\'akosn{\'i}k, \emph{On spaces ${L}^{p(x)}$ and
  ${W}^{k,p(x)}$}, Czechoslovak Math. J \textbf{41} (1991), 592--618.

\bibitem{LB3}
W.~B. Liu and John~W. Barrett, \emph{A further remark on the regularity of the
  solutions of the {$p$}-{L}aplacian and its applications to their finite
  element approximation}, Nonlinear Anal. \textbf{21} (1993), no.~5, 379--387.
  \MR{1237129 (94h:35027)}

\bibitem{LB}
\bysame, \emph{A remark on the regularity of the solutions of the
  {$p$}-{L}aplacian and its application to their finite element approximation},
  J. Math. Anal. Appl. \textbf{178} (1993), no.~2, 470--487. \MR{1238889
  (95a:35016)}

\bibitem{PW}
Andreas Prohl and Weindl Isabelle, \emph{Convergence of an implicit finite
  element discretization for a class of parabolic equations with nonstandard
  anisotropic growth conditions}, http://na.uni-tuebingen.de/preprints.shtml
  (2007).

\bibitem{R}
Michael R{\ocirc{u}}{\v{z}}i{\v{c}}ka, \emph{Electrorheological fluids:
  modeling and mathematical theory}, Lecture Notes in Mathematics, vol. 1748,
  Springer-Verlag, Berlin, 2000.

\bibitem{Sam1}
S.~Samko, \emph{Denseness of {$C\sp \infty\sb 0(\mathbf R\sp N)$} in the
  generalized {S}obolev spaces {$W\sp {M,P(X)}(\mathbf R\sp N)$}}, Direct and
  inverse problems of mathematical physics ({N}ewark, {DE}, 1997), Int. Soc.
  Anal. Appl. Comput., vol.~5, Kluwer Acad. Publ., Dordrecht, 2000,
  pp.~333--342.

\end{thebibliography}

\end{document}